\documentclass[smallextended,referee,envcountsect,]{svjour3}
\smartqed
\usepackage{graphicx}
\usepackage{amssymb,amsfonts,dsfont,amsmath,xcolor}

\begin{document}

\title{Existence of Weak Efficient Solutions of Set-Valued Optimization Problems
\thanks{The first and the second authors were partially supported by the grants from IPM (Nos. 1401490047, 1402460315).}
}

\titlerunning{Existence of Weak Efficient Solutions of Set-Valued Optimization Problems}        

\author{Fatemeh Fakhar,  Hamid Reza Hajisharifi, Zeinab Soltani }


\institute{
Fatemeh Fakhar \at
Department of Pure Mathematics,
 Faculty of Mathematics and Statistics,
  University of Isfahan, Isfahan 81746-73441, Iran.\\
  School of Mathematics, Institute for Research in Fundamental Sciences (IPM),
 P.O. Box 19395-5746, Tehran, Iran.\\
\email{f.fakhar66@gmail.com}
\and
Hamid Reza Hajisharifi (corresponding author)  \at
 Department of Mathematics, Khansar campus,
  University of Isfahan, Isfahan, Iran.\\
School of Mathematics, Institute for Research in Fundamental Sciences (IPM),
 P.O. Box 19395-5746, Tehran, Iran.\\
              \email{h.r.hajisharifi@khc.ui.ac.ir}           
\and
          Zeinab Soltani \at
 Department of Pure Mathematics,
         University of Kashan,
        Kashan, 87317-53153,  I. R. Iran. \\
         \email{z.soltani@kashanu.ac.ir}
}

\date{Received: date / Accepted: date}

\maketitle

\begin{abstract}
In this paper, we consider a new scalarization function for set-valued maps. As the main goal, by using this scalarization function, we obtain some Weierstrass-type theorems for the noncontinuous set optimization problems via the coercivity and noncoercivity conditions. This contribution improves various existing results in the literature.
\end{abstract}

 \section{Introduction}
 Set-valued optimization theory has numerous applications in economics, financial mathematics,  medical engineering, optimal control and several other fields; see for example \cite{CHY,CFFH2023,E,J,K0,L} and the references therein.
The vector approach and the set approach are two main approaches to deal with set-valued optimization problems. The vector approach consists in finding the optimal elements for union of values of the objective set-valued map; see for more details \cite{CHY,E,J,K0,L}. 
 The notion of set-valued optimization with the set criterion which has been called set optimization was first introduced by Kuroiwa \cite{K2,K3,K4}. Recently, many researchers studied the existence of solutions for set-valued optimization problems with the set criterion; see, for example \cite{E,HL,J,K0,K3,K4}.\\
 Most of the set optimization problems have been studied under coercivity conditions. However, recently many researchers have studied set optimization problems without any coercivity assumption, via asymptotic analytic tools; see for example \cite{CFFH2023,HL}.

The main motivation of this paper is to obtain two new Weierstrass-type theorems for a noncontinuous set optimization problem via a coercivity condition and asymptotic analytic tools, by using a new scalarization function. For this purpose, in section 3, we introduce a new scalarization function for set-valued maps. In section 4, we extend the notion of regular-global-inf (rgi) functions \cite{A1} to set-valued maps and we study the class of maps that satisfy this condition. In section 5, as the main results, we have two new Weierstrass-type theorems for noncontinuous set optimization problems. In one of them, we assume the coercivity condition but in the other theorem, there is no coercivity condition and we will use the asymptotic analytic tools.

\section{Preliminaries and notation}
 In this section, we give some  notation and results that will be used in this paper.\\
 We denote the set of real numbers (resp. nonnegative real numbers and nonnegative integers) by $\mathbb{R}$ (resp. $\mathbb{R}_+$ and $\mathbb{N}$). Let $K$ be a nonempty subset of a topological space, the closure, the topological interior and the complement of $K$ are denoted by $\mathrm{cl}(K)$, $\mathrm{int} K$ and $K^c$, respectively.\\
We assume that $X$ is a Hausdorff topological space and $Y$ is a Hausdorff topological vector space ordered by a closed convex cone $P\subseteq Y$ such that:
 $$y_1\preccurlyeq y_2 \Leftrightarrow y_2-y_1\in P.$$
 The convex cone $P$ is said to be solid if $\mathrm{int} P\neq \emptyset$. Whenever $P$ is solid, we have
 $$y_1\prec y_2\Leftrightarrow y_2-y_1\in \mathrm{int} P.$$
 Suppose that $\mathcal{P}(Y)$ is the class of all nonempty subsets of $Y$. For $A,B\in \mathcal{P}(Y)$, the lower set less relation $\leq^l $, which was introduced in \cite{K2,K3}, is defined as follows: $A\leq^l B$ if and only if $B\subseteq A+P$. We say that $A\sim^l B$ if and only if $A\leq^l B$ and $B\leq^l A$ or equivalently $A+P=B+P$.
 We say that  $A<^l B$ if and only if $B\subseteq A+\mathrm{int} P$. It is clear that $\leq^l $ is a preorder relation, although $<^l$ is not necessarily a reflexive relation.

 Throughout the paper, suppose that $P$ is solid and proper; i.e. $\{0\}\neq P\neq Y$, $q\in \mathrm{int} P$, $\mathcal{P}(Y)$ is equipped with the preorder relation $\leq^l $ and $F : X\rightrightarrows Y$ is a set-valued map with nonempty values.\\
A general form of the set optimization problem is usually formulated as follows:
 $$\mathrm{(SOP)}\quad \min F(x) \quad \mbox{subject ~~~to ~~~~} x\in X.$$

 \begin{definition}\cite{HL}
 	An element  $\bar{x}\in X$ is said to be
 	\begin{itemize}
 		\item[(i)] a weakly $l$-efficient solution of (SOP) (briefly, denoted by $\bar{x}\in l$-$WEff(F)$), if for each $x\in X$, $F(x)<^l F(\bar{x})$ implies that $F(\bar{x})<^l F(x)$.
 		\item[(ii)] a strict weakly $l$-efficient solution of (SOP) (briefly, denoted by $\bar{x}\in l$-$
 		SWEff(F)$), if there is no $x\in X$ with $x\neq\bar{x}$ such that $F(x)<^l F(\bar{x})$.
 	\end{itemize}
 \end{definition}
A nonempty subset $A$ of $Y$ is said to be $P$-proper if $A+P\neq Y$, $P$-closed if $A+P$ is closed, $P$-bounded if for each neighborhood $U$ of zero in $Y$ there is some $r>0$ such that $A\subseteq rU + P$ and $P$-compact if any cover of $A$ of the form $\{U_{\alpha}+P:\alpha\in I, U_{\alpha}~\mbox{are open}\}$ admits a finite subcover. By \cite{HR07,L}, every $P$-compact set is $P$-closed and $P$-bounded and by \cite{GMMN}, every $P$-bounded set is $P$-proper.\\
Let $\mathcal{H}$ be the property of sets in $Y$, we say that $F$ is $\mathcal{H}$-valued if for each $x\in X$, $F(x)$ has the property $\mathcal{H}$.\\
In \cite[Lemma 2.2]{CFFH2023}, we see that, if $F$ is $P$-closed-valued and $P$-proper-valued, then
$l$-$WEff(F)=l$-$SWEff(F)$.\\
Similarly to \cite{HL}, we introduce the following definition.
 \begin{definition}
 	 The Colevel set of $F$ at height $B\subseteq Y$ is denoted by  $$Colev_{<^l}(F,B):=\{x\in X: B\nless^l F(x)\}.$$
 	 \end{definition}
 For $y\in Y$, we set $Colev_{<^l}(F,y) :=Colev_{<^l}(F,\{y\})$.

   By \cite{LSS}, we say that, $G:\Lambda\rightrightarrows Y$ is transfer closed on $\Lambda$, if
  	$$\bigcap_{\lambda\in \Lambda}\mathrm{cl} (G(\lambda))=\bigcap_{\lambda\in \Lambda}G(\lambda).$$

 Let $f:X\rightarrow \Bbb R\cup\{+\infty\}$ be a function. This function is said to be proper if there exists $x\in X$ such that $f(x)<+\infty$. When $f$ be a proper function we set $\arg \min f:=\{x\in X: f(x)=\inf f\}$ and, for $\lambda\in \Bbb R$, $[f*\lambda]:=\{x\in X: f(x)*\lambda\}$, where $*\in\{\leq,\geq,<,>,=\}$.

\begin{definition}\cite{A1}
The function $f:X\rightarrow \Bbb R\cup \{+\infty\}$ is called
\begin{itemize}
\item[(i)] lower semicontinuous (lsc) at $x_0\in X$ if,
	$$f(x_0)\leq\liminf_{x\to x_0}f(x).$$
\item[(ii)] regular-global-inf (rgi) at $x_0\in X$ if, $f(x_0)>\inf f$ implies that there exist $U\in\mathcal{N}(x_0)$ (neighborhoods of $x_0$) and $c>\inf f$ such that $\inf f(U)>c$.
\item[(iii)] lsc (rgi) on $X$ if and only if it is lsc (rgi) at every $x\in X$.
\end{itemize}
\end{definition}
Notice that, $f:X\rightarrow \Bbb R\cup \{+\infty\}$ is lsc if and only if $[f\leq \lambda]$ is closed for each $\lambda\in\Bbb R$. Also, every lsc function is rgi, but the converse of this result is not necessarily true; (see \cite{A1}).

Now, we recall the following scalarization function $\psi^q_P:Y\rightarrow \Bbb R$, \cite{GL}, which is particular case of the Gerstewitz scalarization function:
$$\psi^q_P(y):=\sup\{t\in\Bbb R: y\in tq+P\}.$$
This function is $\leq$-increasing ($<$-increasing); i.e., $y_1, y_2  \in Y$
and $y_1 \preccurlyeq y_2 $ ($y_1 \prec y_2$) imply $\psi^q_P (y_1)\leq \psi^q_P(y_2)$ ($\psi^q_P(y_1)< \psi^q_P(y_2)$);
subadditive, positively homogeneous and translation-invariant; i.e., $\psi^q_P (y+\alpha q)=\psi^q_P(y)+\alpha$ for all $y \in Y$, $\alpha \in \Bbb R$.\\
In the following, there are some basic properties of $\psi^q_P$, that we will be used in the sequel.
	\begin{proposition}\label{p3.1}
	Let $\lambda\in\Bbb R$, then we have
	\begin{itemize}
	\item[(i)] $y\in \psi^q_P(y)q+P$,
    \item[(ii)] $[\psi^q_{P}>\lambda]=\lambda q+\mathrm{int} P$,
	\item[(iii)] $[\psi^q_{P}\geq \lambda]=\lambda q+P$,
		\item[(iv)] $[\psi^q_{P}=\lambda]=\lambda q+ (P\setminus \mathrm{int} P)$.
	\end{itemize}
\end{proposition}
\begin{proof}
\begin{itemize}
\item[(i)] Let $y\in Y$. For each $n\in\Bbb N $, there exists $t_n\in \Bbb R$, such that $y\in t_n q+P$ and $t_n\rightarrow \psi^q_P(y)$. Therefore, $t_n q\in y-P$, for each $n\in\Bbb N $. Since $P$ is closed, $\psi^q_P(y)q\in y-P$ and so $y\in \psi^q_P(y)q+P$.
\item[(ii)] See \cite[Proposition 2.1(d)]{GL}.
\item[(iii)] By the definition of $\psi^q_P$, it is clear that $\lambda q+P\subseteq [\psi^q_{P}\geq \lambda]$. Let $y\in [\psi^q_{P}\geq \lambda]$, thus for each $n\in\Bbb N $, $\psi^q_{P}(y) > \lambda- \frac{1}{n}$. Therefore by (ii), $y\in (\lambda- \frac{1}{n})q+ \mathrm{int} P\subseteq (\lambda- \frac{1}{n})q+P$. Hence by closedness of $P$, $y\in \lambda q+P$. Then $[\psi^q_{P}\geq \lambda]=\lambda q+P$.
\item[(iv)] This is the consequence of (ii) and (iii).
\end{itemize}
\end{proof}

Throughout this article, we assume that $(E,\|.\|)$ is a Banach space and $\sigma$ is a topology on $E$ coarser than the norm topology such that $(E,\sigma)$ is a Hausdorff topological vector space. Given a sequence $(x_k)\subseteq E$, we denote by $x_k\xrightarrow{\sigma}x$, its convergence in $\sigma$ topology.\\
We now recall some asymptotic notions. Suppose that $A\subseteq E$. The $\sigma$-asymptotic cone $A^\infty_{\sigma}$, see \cite{GLN,P}, is defined as follows:
$$A^\infty_{\sigma}:=\{d\in E: \exists t_k\rightarrow +\infty, \exists (x_k)\subseteq A : \frac{x_k}{t_k}\xrightarrow{\sigma} d\},$$
with the convention $\emptyset^\infty_{\sigma}=\emptyset$ (for this convention can see \cite{FHS,H1}).

For a sequence of nonempty sets $(A_n)\subseteq E$, the $\sigma$-horizon outer limit is defined as follows, \cite{HL}:
$${\limsup}_{n}^{\sigma,\infty} A_n:=\{u\in E: \exists x_{n_j}\in A_{n_j}, t_{n_j}\to +\infty : \frac{x_{n_j}}{t_{n_j}}\xrightarrow{\sigma} u\}.$$

In the following, we have developed the definition of $qx$-asymptotic function \cite{H1}, for an extended real-valued function that is defined on an  infinite-dimensional normed space.
\begin{definition}\cite{FHS}\label{new asymptotic}
Let $K$ be a nonempty subset of $E$ and let $f:K\rightarrow \Bbb R\cup \{+\infty\}$ be a proper function. The $\sigma$-generalized asymptotic ($\sigma g$-asymptotic) function $f^{\sigma g}: K_{\sigma}^{\infty}\rightarrow \Bbb R\cup \{\pm\infty\}$ is defined by
$$f^{\sigma g}(u) :=\inf\{\lambda: u\in ([f\leq \lambda])^\infty_{\sigma}\}.$$
\end{definition} 

\section{A New Scalarization Function Related to Set-Valued Maps}
By using $\psi^q_P$, we define a new scalarization function $\Psi_F:X\rightarrow \Bbb R\cup\{-\infty\}$ related to the set-valued map $F$, for each $x\in X$, as follows:
$$\Psi_F(x):=\inf_{z\in F(x)}\psi^q_P(z).$$

\begin{remark}\label{Rem.scal}
Note that, if $F$ is a $P$-proper-valued map, the scalarization function $\Psi_F$ can be written by using the extension of the Gerstewitz function introduced by Hern\'{a}ndez and Rodr\'{i}guez-Mar\'{i}n in \cite{HR07}, as follows:
$$\Psi_F(x)=-G_{-q}(\{0\},F(x)),~~\mbox{for~every}~~x\in X.$$ 
\end{remark}

By Proposition \ref{p3.1}, we have the following lemma.
\begin{lemma}\label{Lem 1 Psi}
Let $\lambda\in\Bbb R$, then we have
	\begin{itemize}
\item[(i)]$[\Psi_F\geq \lambda]=\{x\in X: F(x)\subseteq \lambda q+P\}$,
\item[(ii)] $Colev_{<^l}(F,\lambda q)\subseteq [\Psi_F\leq \lambda]$.
\end{itemize}
\end{lemma}
\begin{proof}
\begin{itemize}
\item[(i)] This is a consequence of Proposition \ref{p3.1}(iii).
\item[(ii)] Let $x\in Colev_{<^l}(F,\lambda q)$. Thus $F(x)\nsubseteq \lambda q+\mathrm{int} P$. Therefore there exists $y\in F(x)$ such that $y\notin\lambda q+\mathrm{int} P$. Hence by Proposition \ref{p3.1}(ii), $\psi^q_{P}(y)\leq \lambda$. Then $\Psi_F(x)\leq \lambda$ and $x\in [\Psi_F\leq \lambda]$.
\end{itemize}
\qed
\end{proof}
Now, we introduce the following assumption on $F$.\\

{\bf Assumption $(A)$:}~~$\psi^q_P$ attains its infimum on $F(x)$ for all $x\in X$.\\

\begin{remark}\label{Rem. w-cpt-valued}
Let $F$ be a $P$-compact-valued map. Then by Remark \ref{Rem.scal}, \cite[Proposition 3.4]{HR07} and \cite[Theorem 3.6]{HR07}, $F$ satisfies assumption $(A)$.\\
Also, let $F$ be a compact-valued map. Therefore, by \cite[page 15]{L}, $F$ is $P$-compact-valued. Then $F$ satisfies assumption $(A)$.   

\end{remark}
The following proposition presents a characterization of set-valued maps that satisfy assumption $(A)$. 
\begin{proposition}\label{scalarization}
 If $F$ satisfies assumption $(A)$, then for each $\lambda\in\Bbb R$,
$$Colev_{<^l}(F,\lambda q)=[\Psi_F\leq \lambda].$$
The converse is true, if $F$ is a $P$-bounded-valued map.
\end{proposition}
\begin{proof}
At first, suppose that $F$ satisfies assumption $(A)$ and $\lambda\in \Bbb R$. By Lemma \ref{Lem 1 Psi}(ii) we have, $Colev_{<^l}(F,\lambda q)\subseteq [\Psi_F\leq \lambda]$. Now, let $x\in [\Psi_F\leq \lambda]$. Thus $\Psi_F(x)\leq \lambda$. Since $F$ satisfies assumption $(A)$, there exists $y\in F(x)$ such that $\psi^q_P(y)\leq \lambda$. Therefore by Proposition \ref{p3.1}(ii), $y\notin \lambda q+ \mathrm{int} P$. Then $F(x)\nsubseteq \lambda q+\mathrm{int} P$, and $x\in Colev_{<^l}(F,\lambda q)$.\\
For the converse, by Remark \ref{Rem.scal} and \cite[Theorem 3.6]{HR07}, $\Psi_F(x)>-\infty$ for each $x\in X$.
Now on the contrary, suppose that there exists $x\in X$ such that $\psi^q_P$ doesn't attain its infimum on $F(x)$. Thus, $\psi^q_P(y)>\Psi_F(x)$ for each $y\in F(x)$. Hence, by Proposition \ref{p3.1}(ii), $y\in \Psi_F(x)q+ \mathrm{int} P$ for each $y\in F(x)$. Therefore $\Psi_F(x)q<^lF(x)$. Then, by assumptions we have, $\Psi_F(x)<\Psi_F(x)$, which is a contradiction.
\qed
\end{proof}
Now, we have the following corollary.
\begin{corollary}\label{cor. w-closed bd val.}
Suppose that $X=\Bbb R^n$ and $Y=\Bbb R^m$ are equipped with the usual norm topology, $F$ is $P$-closed-valued and  bounded-valued. Then $F$ satisfies assumption $(A)$.
\end{corollary}
\begin{proof}
Let $\lambda\in \Bbb R$ and let $x\in [\Psi_F\leq \lambda]$. Thus for each $n\in\Bbb N$, there exists $y_n\in F(x)\subseteq F(x)+P$ such that $\psi^q_P(y_n)\leq \lambda+\frac{1}{n}$. Since $F(x)$ is bounded and $F(x)+P$ is closed, there exist $(y_{k_n})\subseteq (y_n)$ and $y\in F(x)+P$ such that $y_{k_n}\to y$. By Proposition \ref{p3.1}(ii), $\psi^q_P$ is a lsc function, therefore $\psi^q_P(y)\leq \lambda$. Hence by Proposition \ref{p3.1}(ii), $y\notin \lambda q+ \mathrm{int} P$. Thus $F(x)+P\nsubseteq \lambda q+ \mathrm{int} P$. Then $F(x)\nsubseteq \lambda q+ \mathrm{int} P$, because, by $F(x)\subseteq \lambda q+ \mathrm{int} P$ we have $F(x)+P\subseteq \lambda q+ \mathrm{int} P$, which is a contradiction. So $x\in Colev_{<^l}(F,\lambda q)$, and we conclude that $[\Psi_F\leq \lambda]\subseteq Colev_{<^l}(F,\lambda q)$. Therefore by Lemma \ref{Lem 1 Psi}(ii), $Colev_{<^l}(F,\lambda q)= [\Psi_F\leq \lambda]$. Since for each $x\in X$, $F(x)$ is bounded, by \cite[Remark 2.3]{GMMN}, $F$ is a $P$-bounded-valued map. Then by Proposition \ref{scalarization}, $F$ satisfies assumption $(A)$.  
\qed
\end{proof}
\begin{remark}\label{rem. osc and loc}
Let $X$ and $Y$ be as Corollary \ref{cor. w-closed bd val.}. By \cite{HL}, if $F$ be $l$-osc and locally bounded (for definitions can see, \cite{HL}), $F$ is $P$-closed-valued and  bounded-valued. Then by Corollary \ref{cor. w-closed bd val.}, $F$ satisfies assumption $(A)$. 
\end{remark}
The following result will be crucial in the proof of the main results of this paper.
\begin{proposition}\label{relation with scalarization}
Let $F$ satisfies assumption $(A)$. Then
$$\arg \min (\Psi_F)\subseteq l\mbox{-}SWEff(F).$$
\end{proposition}
\begin{proof}
Since $F$ satisfies assumption $(A)$, for each $x\in X$, $\Psi_F(x)\in \Bbb R$.\\
Now on the contrary suppose that $\bar{x}\in\arg \min (\Psi_F)$, but $\bar{x}\notin l$-$SWEff(F)$. Thus there exists $x_0\in X$ such that $x_0\neq \bar{x}$ and $F(\bar{x})\subseteq F(x_0)+\mathrm{int} P$. Since $\psi^q_P$ attains its infimum on $F(\bar{x})$, so there exists $\bar{y}\in F(\bar{x})$ such that $\Psi_F(\bar{x})=\psi^q_P(\bar{y})$. On the other hand by $\bar{y}\in F(\bar{x})\subseteq F(x_0)+\mathrm{int} P$, there exists $y_0\in F(x_0)$ such that $\bar{y}\in y_0+\mathrm{int} P$. Also, by Proposition \ref{p3.1}(i), $y_0\in \psi^q_P(y_0)q+P$, hence $\bar{y}\in \psi^q_P(y_0)q+\mathrm{int} P$.
Therefore by Proposition \ref{p3.1}(ii),
$$\psi^q_P(\bar{y})>\psi^q_P(y_0).$$
Then we have
$$\Psi_F(\bar{x})=\psi^q_P(\bar{y})>\psi^q_P(y_0)\geq \Psi_F(x_0),$$
 which is a contradiction.
\qed
\end{proof}
In the following example, we show that in Proposition \ref{relation with scalarization} assumption $(A)$ is a necessary condition and also, in this proposition, the equality does not necessarily hold.
\begin{example}\label{scalar example}
Suppose that $X=\Bbb R$ and $Y=\Bbb R^2$ are equipped with the usual norm topology and $P=\Bbb R_{+}^2:=\{(x,y)\in \Bbb R^2: x,y\geq 0\}$.
\begin{itemize}
\item[(i)] Assume that $q=(1,1)$ and $A=\{(a,b)\in \Bbb R^2:\frac{1}{a}\leq b~,~ a>0\}$. Let $F:\Bbb R\rightrightarrows\Bbb R^2$ be defined by $F(x)=A$ for $x\neq 0$ and $F(0)=\{(0,0)\}$. Thus for each $x\in \Bbb R$, $\Psi_F(x)=0$ and $argmin(\Psi_F)=\Bbb R$. Also, we have $l$-$SWEff(F)=\{0\}$. Then $\arg \min (\Psi_F)\nsubseteq l\mbox{-}SWEff(F)$. Note that, $\psi^q_P$ doesn't attain its infimum on $F(x)$, for each $x\in \Bbb R\setminus\{0\}$.
\item[(ii)] Assume that $q=(\frac{1}{2},\frac{1}{2})$ and $A=\{(a, b)\in \Bbb R^2:~~a, b\in [3,4]\}$. Let $F:\Bbb R\rightarrow \Bbb R^2$ be defined by
	$$F(x)=\left\{
\begin{array}{ll}
	\{(x,1-x)\},& 0\leq x\leq 1,\\
	A,& o. w.\\
\end{array}\right.$$
Hence
$$\Psi_F(x)=\left\{
\begin{array}{ll}
	2x,&~ 0\leq x\leq \frac{1}{2},\\
2-2x,&~ \frac{1}{2}< x\leq 1,\\
	6,&~ o. w.\\
\end{array}\right.$$

Therefore $argmin (\Psi_F)=\{0,1\}$ and $l$-$SWEff(F)=[0,1]$. Then $l\mbox{-}SWEff(F)\neq\arg \min (\Psi_F)$. However, by Remark \ref{Rem. w-cpt-valued}, $F$ satisfies assumption $(A)$.
\end{itemize}
\end{example}

\section{A New Continuity Notion for Set-Valued Maps}
In this section, inspired by the notion of rgi functions \cite{A1}, we extend this notion to set-valued maps.\\
For the set-valued map $F$ we set
$$M_F^q:=\sup\{t\in\Bbb R: F(X)\subseteq tq+P\},$$
by the convention, $\sup\emptyset=-\infty$. We can see $M_F^q=\inf_{z\in F(X)}\psi^q_P(z)=\inf \Psi_F$. Notice that, for $\lambda>M_F^q$, $Colev_{<^l}(F,\lambda q)$ is a nonempty set.\\
Now we have  the definition of rgi maps in set-valued case.
\begin{definition}\label{rgi set valued}
We say that the set-valued map $F$ is $q$-regular-global-inf ($q$-srgi) if and only if $x_0\in (Colev_{<^l}(F,\lambda q))^c$ for some $\lambda>M_F^q$ implies that there exist $U\in \mathcal{N}(x_0)$ and $r>M_F^q$ such that $U\subseteq (Colev_{<^l}(F,rq))^c$.
\end{definition}
By the following proposition we have a characterization for $q$-srgi maps.
\begin{proposition}\label{P. tlc}
Let the set-valued map $G_F^q:(M_F^q,+\infty)\rightrightarrows X$ be defined by $G_F^q(\lambda)=Colev_{<^l}(F, {\lambda q})$ for each $\lambda>M_F^q$. Then $F$ is a $q$-srgi map if and only if $G_F^q$ is transfer closed on $(M_F^q,+\infty)$, that is,
	$$\bigcap_{\lambda>M_F^q}\mathrm{cl}(Colev_{<^l}(F,\lambda q))=\bigcap_{\lambda>M_F^q}Colev_{<^l}(F,\lambda q).$$
\end{proposition}
\begin{proof}
Suppose that $F$ is a $q$-srgi map and $x_0\notin \bigcap_{\lambda>M_F^q}Colev_{<^l}(F,\lambda q)$. Then, there exists $\lambda_0>M_F^q$ such that $x_0\notin Colev_{<^l}(F,\lambda_0 q)$. From the definition, there exist $U\in \mathcal{N}(x_0)$ and $r>M_F^q$ such that $U\subseteq (Colev_{<^l}(F,rq))^c$. Hence
$$U\bigcap Colev_{<^l}(F,rq)=\emptyset.$$
Thus, $x_0\notin \mathrm{cl}(Colev_{<^l}(F,rq))$ and so, $x_0\notin\bigcap_{\lambda>M_F^q}\mathrm{cl}(Colev_{<^l}(F,\lambda q))$.
	Therefore, we have
$$\bigcap_{\lambda>M_F^q}\mathrm{cl}(Colev_{<^l}(F,\lambda q))=\bigcap_{\lambda> M_F^q}Colev_{<^l}(F,\lambda q).$$
	Conversely, suppose that $G_F^q$ is transfer closed on $(M_F^q,+\infty)$, $x_0\in X$ and for some $\lambda> M_F^q$, $x_0\notin Colev_{<^l}(F,\lambda q)$. Since
$$\bigcap_{\lambda>M_F^q}\mathrm{cl}(Colev_{<^l}(F,\lambda q))=\bigcap_{\lambda> M_F^q}Colev_{<^l}(F,\lambda q),$$
 for some $r> M_F^q$ we have $x_0\notin \mathrm{cl}(Colev_{<^l}(F,rq))$. Thus, there exists $U\in \mathcal{N}(x_0)$ such that $U \cap Colev_{<^l}(F,rq)=\emptyset$. Therefore $U\subseteq (Colev_{<^l}(F,rq))^c$. Then $F$ is a $q$-srgi map.
\qed
\end{proof}
\begin{corollary}\label{Cr. rq closed}
Let for each $\lambda>M_F^q$, $Colev_{<^l}(F,\lambda q)$ be a closed set. Then $F$ is $q$-srgi.
\end{corollary}
Recall that, by \cite{HL}, $F$ is said to be $l$-colevel closed, if $Colev_{<^l}(F,B)$ be a closed set for each $B\subseteq Y$.
\begin{corollary}\label{cor2}
Let $F$ be $l$-colevel closed. Then $F$ is a $q$-srgi map.
\end{corollary}
In the following examples we show that the converse statement of Corollary \ref{Cr. rq closed} does not necessarily holds.
\begin{example}
Suppose that $X=Y=\ell^\infty$ are equipped with the usual norm topology, $P=\ell_{+}^{\infty}=\{(x_i)\in \ell^\infty:~x_i\geq 0,~\forall i\in \Bbb N\}$, $q=(1,1,...)$ and $D=\{rq: r\in\Bbb R\}$. Let $F: \ell^{\infty}\rightrightarrows \ell^{\infty}$ be defined as follows:
$$F(x)=\left\{
\begin{array}{ll}
	\{|r|q\}, \quad\quad\quad\quad & x=rq,\\
	\{tq:t\geq 1\}, & x\in D^c.
\end{array}\right.$$
Therefore, $M_q^F=0$ and
$$Colev_{<^l}(F,\lambda q)=\left\{
\begin{array}{ll}
	D^c\cup\{rq:|r|\leq \lambda\}, & \lambda\geq 1,\\
	\{rq:|r|\leq \lambda\}, &  0<\lambda<1.\\
\end{array}\right.$$
Hence, for $\lambda\geq 1$, $Colev_{<^l}(F,\lambda q)$ is not closed, but
$$\bigcap_{\lambda>0}Colev_{<^l}(F,\lambda q)=\bigcap_{\lambda>0} \mathrm{cl}(Colev_{<^l}(F,\lambda q))=\{0\}.$$
\end{example}
\begin{example} Suppose that $X=Y=\Bbb R$ are equipped with the usual norm topology, $P=\Bbb R_+$ and $q=1$. Let $F: \Bbb R\rightrightarrows \Bbb R$ be an interval function; i.e., $F(x) = [F^L(x), F^U(x)]$, where $F^L, F^U:\Bbb R  \to \Bbb R$ are two proper functions with $F^L(x)\leq F^U(x)$ for all $x\in \Bbb R$.  It's easy to see that $M_F^q=\inf {F^L}$ and for each $r\in \Bbb R$,
\begin{equation}\label{EQ1}
Colev_{<^l}(F,rq)=\{x\in \Bbb R: F(x)\leq^l \{rq\}\}=[F^L\leq r].
\end{equation}
Therefore $F$ is $q$-srgi if and only if $F^L$ is rgi.
Suppose that $F^L, F^U: \Bbb R \to \Bbb R$ are defined as
 $$F^L(x)=\left\{
		\begin{array}{ll}
			|x|,&~~ x < 1,\\
			2x,& ~~x\geq 1,
		\end{array}\right.$$
and $F^U(x)=|2x|$, for each $x\in\Bbb R$. Thus $F^L$ is rgi but it is not lsc and $Colev_{\leq^l}(F,\frac{3}{2}q)$ is not closed. Indeed, by (\ref{EQ1}), $Colev_{<^l}(F,\frac{3}{2}q)=[F^L\leq \frac{3}{2}]=(-\infty, 1)$. Then $F$ is $q$-srgi, but by (\ref{EQ1}), it is not $P$-lsc (for definition can see, \cite[Definition 3.1.22(i)]{K0}) and for some $\lambda>M_F^q$, $Colev_{<^l}(F,\lambda q)$ is not a closed set.
\end{example}
In the following proposition we have another characterization for $q$-srgi maps, which will be used in the next section.
\begin{proposition}\label{Prop.rgi}
The set-valued map $F$ is $q$-srgi if and only if $\Psi_F$ is rgi.
\end{proposition}
\begin{proof}
Let $F$ be a $q$-srgi map and let $\Psi_F(x_0)>\lambda$ for some $\lambda>\inf\Psi_F=M_F^q$. By Lemma \ref{Lem 1 Psi}(ii), we have $x_0\notin Colev_{<^l}(F,\lambda q)$. Hence there exist $U\in \mathcal{N}(x_0)$ and $r>M_F^q$ such that $U\subseteq (Colev_{<^l}(F,rq))^c$. Therefore for each $x\in U$, $x\notin Colev_{<^l}(F,rq)$. Then we conclude that $F(x)\subseteq rq+ \mathrm{int}P$, for each $x\in U$. Thus by Lemma \ref{Lem 1 Psi}(i) for each $x\in U$, $\Psi_F(x)\geq r$. Hence
$$\inf\Psi_F(U)=\inf_{x\in U} \Psi_F(x)\geq r>M_F^q=\inf\Psi_F,$$
and so $\Psi_F$ is rgi.\\
Conversely, let for some $\lambda>M_F^q$ and $x_0\in X$, $x_0\notin Colev_{<^l}(F,\lambda q)$, that is, $F(x_0)\subseteq \lambda q+\mathrm{int}P$. By Lemma \ref{Lem 1 Psi}(i), $\Psi_F(x_0)\geq \lambda$. Hence, there exists $M_F^q<\lambda'<\lambda$ such that $\Psi_F(x_0)>\lambda'$. Since $\Psi_F$ is rgi, there exists $r\in \Bbb R$ and $U\in \mathcal{N}(x_0)$ such that
$$\inf\Psi_F(U)=\inf_{x\in U} \Psi_F(x)>r>\inf\Psi_F=M_F^q.$$
Therefore for each $x\in U$, $\Psi_F(x)>r$. Thus by Lemma \ref{Lem 1 Psi}(ii), $U\subseteq (Colev_{<^l}(F,rq))^c$. Then $F$ is a $q$-srgi map.
\qed
\end{proof}

\section{Some New Existence Results of Strict Weakly $l$-Efficient Solution of (SOP)}
In this section, we obtain some existence results of strict weakly $l$-efficient solution of (SOP).\\ 
At first, we define a coercivity condition for set-valued maps as the following.
\begin{definition}
 We say that $F$ satisfies $q$-global-inf-coercive-condition ($q$-sgicc), if there exists $\lambda>M_F^q$ such that $Colev_{<^l}(F,\lambda q)$ is a relatively compact set.
\end{definition}
Now, we have the following proposition.
\begin{proposition}\label{Pro.wcpt to gicc}
Let $F$ satisfies assumption $(A)$ and let there exists $x_0\in X$ such that $Colev_{<^l}(F,F(x_0))$ is relatively compact. Then $x_0\in l$-$SWEff(F)$ or $F$ satisfies $q$-sgicc.
\end{proposition}
\begin{proof}
Let $x_0\notin l$-$SWEff(F)$, thus by Proposition \ref{relation with scalarization}, $x_0\notin \arg \min (\Psi_F)$. Hence $\Psi_F(x_0)>M_F^q$ and there exists $\lambda>M_F^q$ such that $\Psi_F(x_0)>\lambda$. Therefore by Lemma \ref{Lem 1 Psi}(ii), $F(x_0)\subseteq \lambda q+\mathrm{int}P$.\\
Now assume that $x\notin Colev_{<^l}(F, F(x_0))$, thus $F(x_0)<^l F(x)$. Therefore, since $F(x_0)\subseteq \lambda q+\mathrm{int}P$ we have $x\notin Colev_{<^l}(F,\lambda q)$. Hence $Colev_{<^l}(F,\lambda q)\subseteq Colev_{<^l}(F, F(x_0))$. Then $Colev_{<^l}(F,\lambda q)$ is relatively compact and $F$ satisfies $q$-sgicc.
\qed
\end{proof}
In the following example we show that there exists a set-valued map that satisfies $q$-sgicc but for each $x\in X$, $Colev_{<^l}(F, F(x))$ is not a relatively compact set.
\begin{example}\label{e14}
Suppose that $X=Y=\Bbb R^2$ are equipped with the usual norm topology, $P=\Bbb R^2_+$ and $q=(1,1)$. Let $F:\Bbb R^2\rightrightarrows \Bbb R^2$ be defined by
	$$F(x,y)=\left\{
	\begin{array}{ll}
		(-3,2)+B_{\Bbb R^2},& (x,y)=(1,0),\\
		(|x|,|y|)+B_{\Bbb R^2},& (x,y)\neq (1,0),\\
	\end{array}\right.$$
where $B_{\Bbb R^2}$ is the unit closed ball in $\Bbb R^2$. Then $M_F^q=-4$ and $Colev_{<^l}(F,-3q)=\{(1,0)\}$. Thus $F$ is $q$-sgicc. But for each $(x,y)\in \Bbb R^2$, $Colev_{<^l}(F, F((x,y)))$ is not relatively compact.
\end{example}
In the following, as the first main result of this paper, we have a new Weierstrass-type theorem for a set-valued map that satisfies the coercivity condition $q$-sgicc.
\begin{proposition}\label{first theorem}
Let $F$ be a $q$-srgi map that satisfies assumption $(A)$ and $q$-sgicc.  If $F(X)$ is $P$-bounded, then $l$-$SWEff(F)$ is nonempty.
\end{proposition}
\begin{proof}
By Remark \ref{Rem.scal} and \cite[Theorem 3.6]{HR07}, $\Psi_F$ is bounded from below. Also, by Proposition \ref{Prop.rgi}, $\Psi_F$ is a rgi function. On the other hand by Proposition \ref{scalarization}, for each $\lambda\in\Bbb R$,
$$Colev_{<^l}(F,\lambda q)=[\Psi_F\leq \lambda].$$
Therefore by assumption, there exists $\lambda\in\Bbb R$ such that $[\Psi_F\leq \lambda]$ is a relatively compact set. Thus by \cite[Theorem 16]{A1}, $\arg \min (\Psi_F)$ is a nonempty set. Then by Proposition \ref{relation with scalarization}, $l$-$SWEff(F)$ is nonempty.
\qed
\end{proof}
By Proposition \ref{Pro.wcpt to gicc} and Proposition \ref{first theorem} we have the following corollary.
\begin{corollary}\label{cor4}
Let $F$ be $q$-srgi that satisfies assumption $(A)$ and let there exists $x_0\in X$ such that $Colev_{<^l}(F,F(x_0))$ is relatively compact.  If $F(X)$ is $P$-bounded, then $l$-$SWEff(F)$ is a nonempty set.
\end{corollary}
By Remark \ref{rem. osc and loc}, the following corollary is an extension of \cite[Theorem 3]{HL}.
\begin{corollary}
Suppose that $X=\Bbb R^n$ and $Y=\Bbb R^m$ are equipped with the usual norm topology and $F:\Bbb R^n\rightrightarrows \Bbb R^m$ is $l$-colevel closed, $P$-closed-valued  and bounded-valued. Assume that for some $x_0 \in X$, $Colev_{<^l}(F, F(x_0))$ is bounded and $F(Colev_{<^l}(F, F(x_0)))$ is  a $P$-bounded set. Then $l$-$SWEff(F)$ is nonempty and compact.
\end{corollary}
\begin{proof}
Let $K=Colev_{<^l}(F, F(x_0))$ and let $F|_K: K\rightrightarrows \Bbb R^m$. By Corollary \ref{cor. w-closed bd val.}, Corollary \ref{cor2} and Corollary \ref{cor4}, $l$-$SWEff(F|_K)$ is a nonempty set. Similarly to the proof of \cite[Theorem 3]{HL}, we have $l$-$SWEff(F|_K)\subseteq l$-$SWEff(F)$, then $l$-$SWEff(F)$ is a nonempty set. Also, since $F$ is $l$-colevel closed and $Colev_{<^l}(F, F(x_0))$ is a bounded set, by $l$-$SWEff(F)=\cap_{x\in X}Colev_{<^l}(F, F(x))$, \cite{CFFH2023}, $l$-$SWEff(F)$ is compact.
\qed
\end{proof}
By the following example we see that, Proposition \ref{first theorem} is not consequence of the results in \cite{HR07,KKS17}.
\begin{example}
Suppose that $X=\Bbb R$ and $Y=\Bbb R^2$ are equipped with the usual norm topology, $P=\{(x,y)\in\Bbb R^2: -x\leq y\leq x\}$ and $q=(1,0)$. Let $A=\{(x,y)\in \Bbb R^2: x\geq 0, 0\leq y\leq 2\}\setminus \{(0,0)\}$ and let $F:\Bbb R\rightrightarrows \Bbb R^2$ be defined by
	$$F(x)=\left\{
	\begin{array}{ll}
		A,& x=0,\\
		(1,0),& x\neq 0.\\
	\end{array}\right.$$
Therefore
$$\Psi_F(x)=\left\{
	\begin{array}{ll}
		-2,& x=0,\\
		1,& x\neq 0,\\
	\end{array}\right.$$
and $Colev_{<^l}(F, \frac{1}{2}q)=\{0\}$. Then, we can see, $F$ satisfies all the conditions of Proposition \ref{first theorem} and $0\in l$-$SWEff(F)$. But, $A$ is not a $P$-closed set. Hence $F$ is not a $P$-closed-valued map. Then, we can not apply the results in \cite{HR07,KKS17}.
\end{example}

Now, in the following we extend the notion of asymptotic functions in Definition \ref{new asymptotic} to set-valued maps and by using this concept we present new optimality conditions for the noncoercive (SOP). Since the $qx$-asymptotic function is different from other asymptotic functions, then the following asymptotic function is different from other asymptotic functions in the set-valued case.\\
In the sequel, suppose that $K$ is a nonempty subset of the Banach space $(E,\|.\|)$, $B_E$ is the unit closed ball in $(E,\|.\|)$ and for each $n\in \Bbb N$, $K_n=\{x\in K: \|x\|\leq n\}$.
\begin{definition}\label{Def5.2}
For the set-valued map $F: K\rightrightarrows Y$, we define the  $qx$-$\sigma$-asymptotic function $F^{G,\infty}_{\sigma}:K_{\sigma}^{\infty}\rightarrow \Bbb R \cup\{\pm\infty\}$, as follows:
$$F^{G,\infty}_{\sigma}(u)=\inf\{\lambda: u\in (Colev_{<^l}(F,\lambda q))^{\infty}_{\sigma}\}.$$
\end{definition}
\begin{remark}\label{R1}
By Proposition \ref{scalarization} we can see, if $F: K\rightrightarrows Y$ satisfies assumption $(A)$, then $F^{G,\infty}_{\sigma}(u)=(\Psi_F)^{\sigma g}(u)$ for all $u\in K_{\sigma}^{\infty}$.  Therefore by \cite[Proposition 2]{FHS}, if $F: K\rightrightarrows Y$ satisfies assumption $(A)$, then $F^{G,\infty}_{\sigma}(u)=\inf\{\liminf_{n\to\infty}\Psi_F(t_nd_n): t_n\to +\infty, d_n \xrightarrow{\sigma} u\}$ for all $u\in K_{\sigma}^{\infty}$. 
	\end{remark}
In the following example, we illustrate Definition \ref{Def5.2}.
\begin{example}
Suppose that $E=Y=\Bbb R$ are equipped with the usual norm topology, $\sigma=s$ is the usual norm topology, $P=\Bbb R_+$ and $q=1$. Let $F: \Bbb R\rightrightarrows \Bbb R$ be defined by
		$$F(x)=\left\{
		\begin{array}{ll}
				[x, x+1],& x\geq 0,\\
	
			[-1, 0],& x< 0.\\
			\end{array}\right.$$
Thus we have $M_F^q=-1$, and we can see that
		$$Colev_{<^l}(F,\lambda q)=\left\{
		\begin{array}{ll}
			\{x:~x< 0\},& -1<\lambda<0,\\
			\{x:~x< 0\}\cup \{x:~x\geq \lambda\},& \lambda\geq 0.
		\end{array}\right.$$
Therefore
$$(Colev_{<^l}(F,\lambda q))^\infty_{s}=\left\{
		\begin{array}{ll}
			\{x:~x\leq 0\},& -1<\lambda<0,\\
			\Bbb R,& \lambda\geq 0.
		\end{array}\right.$$	
Then
$$F^G_{s}(u)=\left\{
		\begin{array}{ll}
				0,& u> 0,\\
			-1,& u\leq 0.\\
			\end{array}\right.$$
\end{example}
In the following proposition we have some properties of this asymptotic function.
\begin{proposition}\label{P. property qr}
For the set-valued map $F: K\rightrightarrows Y$ we have
\begin{itemize}
\item[(i)] If $F$ satisfies assumption $(A)$, then $\inf F^{G,\infty}_{\sigma}(u)=F^{G,\infty}_{\sigma}(0)=M_F^q$.
\item[(ii)] $F^{G,\infty}_{\sigma}$ is positively homogeneous of degree zero, that is, $F^{G,\infty}_{\sigma}(tu)=F^{G,\infty}_{\sigma}(u)$ for all $t>0$ and $u\in K_{\sigma}^{\infty}$.

Moreover,
\item[(iii)] Let $F_1,F_2:K\rightrightarrows Y$ be two set-valued maps with nonempty values. Then
		\begin{center}
			$F_1\leq^l F_2$ implies that $(F_1)^G_{\sigma}\leq (F_2)^G_{\sigma}$.
		\end{center}
\end{itemize}
\end{proposition}
\begin{proof}
$(i)$ If $Colev_{<^l}(F,\lambda q)\neq\emptyset$, then $0\in (Colev_{<^l}(F,\lambda q))^\infty_{\sigma}$. Thus, $F^{G,\infty}_{\sigma}(u)\geq F^{G,\infty}_{\sigma}(0)$ for all $u\in K_{\sigma}^{\infty}$ and therefore, $\inf F^{G,\infty}_{\sigma}(u)=F^{G,\infty}_{\sigma}(0)$. Also, by Remark \ref{R1} we obtain $(\Psi_F)^{\sigma g}(0)=F^{G,\infty}_{\sigma}(0)$. On the other hand, by Definition \ref{new asymptotic}, $(\Psi_F)^{\sigma g}(0)\geq M_F^q$. Since for each $\lambda>M_F^q$, $Colev_{<^l}(F,\lambda q)\neq\emptyset$, we have $0\in (Colev_{<^l}(F,\lambda q))^\infty_{\sigma}$, thus $(\Psi_F)^{\sigma g}(0)\leq M_F^q$. Hence $M_F^q=(\Psi_F)^{\sigma g}(0)$. Then, we conclude that $\inf F^{G,\infty}_{\sigma}(u)=F^{G,\infty}_{\sigma}(0)=M_F^q$.\\
$(ii)$ Since $(Colev_{<^l}(F,\lambda q))^\infty_{\sigma}=\frac{1}{t}(Colev_{<^l}(F,\lambda q))^\infty_{\sigma}$ for all $t>0$. Thus, $F^{G,\infty}_{\sigma}(tu)=F^{G,\infty}_{\sigma}(u)$ for all $t>0$ and $u\in K_{\sigma}^{\infty}$.\\
$(iii)$ Let $F_1\leq^l F_2$, thus for each $\lambda\in \Bbb R$ we have $Colev_{<^l}(F_2,\lambda q)\subseteq Colev_{<^l}(F_1,\lambda q)$. Then $(F_1)^G_{\sigma}\leq (F_2)^G_{\sigma}$.\\
\qed
\end{proof}
In order to obtain the second main result of this paper, we need the following theorem, which is a corollary of \cite[Theorem 2]{FHS}. 
\begin{theorem}\label{main theorem in minimization}
 Let $K$ be $\sigma$-closed and convex and let $f: K\to\Bbb R$ be bounded from below and $\sigma$-rgi (rgi with respect to the $\sigma$ topology) on $K_n$ for all $n\in \Bbb N$. Suppose that $B_E$ is $\sigma$-closed and sequentially compact with respect to the $\sigma$ topology and for every $u\in K_{\sigma}^{\infty}\setminus \{0\}$, $f^{\sigma g}(u)>f^{\sigma g}(0)$. Then $f$ has a minimum  and $(\arg \min(f))^\infty_{\sigma}=\{0\}$, if the following condition holds.
\begin{itemize}
\item[$(K_{\sigma m})$] If there exists an unbounded sequence $(x_n)$ such that for each $n\in \Bbb N$, $x_n\in \arg \min(f{|_{K_n}})$, then there exist $(n_k)\subseteq (n)$, $(w_k)\subseteq K$ and $d\in K_{\sigma}^{\infty}\setminus \{0\}$ such that for each $k\in \Bbb N$, $w_k\in \arg \min(f{|_{K_{n_k}}})$ and $\frac{w_k}{\|w_k\|}\xrightarrow{\sigma} d$.
\end{itemize}
\end{theorem}
\begin{proof}
Let $m\in\Bbb N$ and $(x_n)\subseteq K_m$ be a minimizing sequence for $f{|_{K_m}}$. Since $B_E$ is sequentially compact with respect to the $\sigma$ topology, there exist $(x_{n_k})\subseteq (x_n)$ and $x_0\in K_m$ such that $x_{n_k}\to x_0$. Hence by \cite[Theorem 2]{A1}, $x_0\in \arg \min(f{|_{K_m}})$ and $\arg \min(f{|_{K_m}})$ is a nonempty set. Also, by \cite[Theorem 11]{A1}, $\arg \min(f{|_{K_m}})$ is $\sigma$-closed. Then by the similar proof of \cite[Theorem 3]{FHS}, we conclude the result by \cite[Theorem 2]{FHS}.
\qed 
\end{proof}

In the following, as another main result of this paper, we have an existence result for (SOP) with some noncoercive conditions.
\begin{theorem}\label{thm1}
Suppose that $K$ is $\sigma$-closed and convex. Let $F: K\rightrightarrows Y$ satisfies assumption $(A)$ and let $B_E$ be $\sigma$-closed and sequentially compact with respect to the $\sigma$ topology. Suppose that $F(X)$ is $P$-bounded and for every $n\in\Bbb N$, $F$ is $\sigma q$-srgi ($q$-srgi with respect to the $\sigma$ topology) on $K_n$ and
\begin{equation}\label{12}
 F^{G,\infty}_{\sigma}(u)>F^{G,\infty}_{\sigma}(0),~~~~~ for~ all ~u\in K_{\sigma}^{\infty}\setminus \{0\}.
\end{equation}
Then $l$-$SWEff(F)$ is a nonempty subset of $K$, if the following condition holds.
\begin{itemize}
\item[$(K_q^{set})$] For any $(\lambda_n)\subseteq \Bbb R$ and any unbounded sequence $(x_n)\subseteq K$ that satisfy $\lambda_n\downarrow M_F^q$, $x_n\in Colev_{<^l}(F,\lambda_nq)\bigcap K_n$ for each $n\in\Bbb N$, and $\|x_n\|\rightarrow +\infty$, there exist $(x_{n_k} )\subseteq (x_n)$, $(w_k)\subseteq K$ and $d\in K_{\sigma}^{\infty}\setminus \{0\}$
such that $F(w_k)\leq^lF(x_{n_k} )$ and $\|w_k\|\leq\|x_{n_k}\|$ for $k$ large enough, and $\frac{w_k}{\|w_k\|}\xrightarrow{\sigma} d.$
	\end{itemize}
\end{theorem}
\begin{proof}
At first, we show that $\Psi_F: K\to \Bbb R$ satisfies condition $(K_{\sigma m})$. For the proof of this purpose, let $(x_n)$ be an unbounded sequence such that for each $n\in \Bbb N$, $x_n\in \arg \min({\Psi_F}{|_{K_n}})$, then there exists $(\lambda_n)\subseteq\Bbb R$ such that $\lambda_n\downarrow M_F^q$ and $\Psi_F(x_n)<\lambda_n$, for each $n\in\Bbb N$. Therefore by Proposition \ref{scalarization}, for every $n\in\Bbb N$ we have $x_n\in Colev_{<^l}(F,\lambda_nq)\bigcap K_n$.\\
Thus, by condition $(K_q^{set})$ there exist $(x_{n_k})\subseteq (x_n)$, $(w_k)\subseteq K$ and $d\in K_{\sigma}^{\infty}\setminus \{0\}$
 such that $F(w_k)\leq^lF(x_{n_k})$ and $\|w_k\|\leq\|x_{n_k}\|$ for $k$ large enough, and $\frac{w_k}{\|w_k\|}\xrightarrow{\sigma} d$. Since $\psi^q_P$ is $\leq$-increasing, we can see that $\Psi_F(w_k)\leq \Psi_F(x_{n_k})$, for $k$ large enough. Indeed, since $F(w_k)\leq^lF(x_{n_k})$, for each $z\in F(x_{n_k})$ there exists $z'\in F(w_k)$ such that $z'\preccurlyeq z$, hence $\psi^q_P(z')\leq \psi^q_P(z)$, and then $\Psi_F(w_k)\leq \Psi_F(x_{n_k})$.\\
Therefore for $k$ large enough, $w_{k}\in \arg \min({\Psi_F}{|_{K_{n_k}}})$ and $\frac{w_k}{\|w_k\|}\xrightarrow{\sigma} d$. Then $\Psi_F$ satisfies condition $(K_{\sigma m})$.\\
Now by Proposition \ref{Prop.rgi}, for all $n\in \Bbb N$, $\Psi_F$ is $\sigma$-rgi on $K_n$ and by Remark \ref{R1} $(\Psi_F)^{\sigma g}(u)>(\Psi_F)^{\sigma g}(0)$ for all $u\in K_{\sigma}^{\infty}\setminus \{0\}$. Thus by Theorem \ref{main theorem in minimization}, $\arg \min (\Psi_F)$ is a nonempty set. Then by Proposition \ref{relation with scalarization}, $l$-$SWEff(F)$ is nonempty.
\qed
\end{proof}
\begin{remark}
Assumption $(K_q^{set})$ is an extension of the well-known assumption proposed by the first time in \cite{BBGT}, in the context of existence results for noncoercive minimization problems. Since it is a weak theoretical assumption and extremely useful for providing existence results in continuous optimization, it was generalized to different convex and nonconvex problems as, multiobjective optimization problems and equilibrium problems among others; see \cite{A,FB,IS} for instance.
\end{remark}
In the following, we have some examples of set-valued maps that satisfy condition $(K_q^{set})$.
\begin{example}\label{exp 6}
\begin{itemize}
\item[(i)] Suppose that $(E,\|.\|)$ is a finite dimensional Banach space, $\sigma=s$ and $F: K\rightrightarrows Y$ is a set-valued map. Then $F$ satisfies condition $(K_q^{set})$. Indeed, let $(\lambda_n)\subseteq \Bbb R$ and $(x_n)\subseteq K$ satisfy $\lambda_n\downarrow M_F^q$, $x_n\in Colev_{<^l}(F,\lambda_nq)\bigcap K_n$ for any $n\in\Bbb N$, and $\|x_n\|\rightarrow +\infty$. Therefore there exist $(x_{n_k})\subseteq (x_n)$ and $d\in K_{s}^{\infty}\setminus \{0\}$ such that $\frac{x_{n_k}}{\|x_{n_k}\|}\xrightarrow{s} d$. It is sufficient to set, for each $k\in\Bbb N$, $w_k=x_{n_k}$.
\item[(ii)] Let $F: K\rightrightarrows Y$ satisfies $q$-sgicc with respect to the $\sigma$ topology, where $\sigma$ is the usual norm or the weak topology on $(E,\|.\|)$. Then $F$ satisfies condition $(K_q^{set})$, because there exists $\lambda>M_F^q$ such that $Colev_{<^l}(F,\lambda q)$ is bounded. Let $(\lambda_n)\subseteq \Bbb R$ and $\lambda_n\downarrow M_F^q$. Then there exists $N\in \Bbb N$ such that for $n\geq N,$ $Colev_{<^l}(F,\lambda_nq)\subseteq Colev_{<^l}(F,\lambda q)$. Thus there exist $M>0$ such that for each $x\in Colev_{<^l}(F,\lambda_nq)$ with $n\geq N$, we have $\|x\|\leq M$. Therefore there is no unbounded sequence $(x_n)$ such that $x_n\in Colev_{<^l}(F,\lambda_nq)\bigcap K_n$, for any $n\in\Bbb N$. Then $F$ satisfies condition $(K_q^{set})$.
\end{itemize}
\end{example}

We have the following proposition about the optimality condition (\ref{12}).
\begin{proposition}\label{Prop horizon}
Suppose that $(E,\sigma)$ is first countable and $F: K\rightrightarrows Y$ satisfies assumption $(A)$. Then $F$ satisfies the optimality condition (\ref{12}) if and only if for each $\lambda_n\downarrow M_F^q$,
$${\limsup}_n^{\sigma,\infty} Colev_{<^l}(F,\lambda_nq)=\{0\}.$$ 
\end{proposition}
\begin{proof}
Since $F$ satisfies assumption $(A)$, by Remark \ref{R1}, $F^{G,\infty}_{\sigma}(u)=(\Psi_F)^{\sigma g}(u)$ for all $u\in K_{\sigma}^{\infty}$ and by Proposition \ref{P. property qr}(i), $\inf F^{G,\infty}_{\sigma}(u)=F^{G,\infty}_{\sigma}(0)=M_F^q$.\\
Let for each $\lambda_n\downarrow M_F^q$, ${\limsup}_n^{\sigma,\infty} Colev_{<^l}(F,\lambda_nq)=\{0\}$ and there exists $u\in K_{\sigma}^{\infty}$ such that $F^{G,\infty}_{\sigma}(u)=F^{G,\infty}_{\sigma}(0)$. Therefore $(\Psi_F)^{\sigma g}(u)=M_F^q$. Hence by Remark \ref{R1} we have
$$\inf\{\liminf_{n\to\infty}\Psi_F(t_nd_n): t_n\to +\infty, d_n \xrightarrow{\sigma} u\}=M_F^q.$$
Thus for some $\lambda_n\downarrow M_F^q$ there exist $t_n\to +\infty$ and $d_n \xrightarrow{\sigma} u$ such that 
$$\Psi_F(t_nd_n)\leq \lambda_n,~~~~~\mbox{for~all}~n\in\Bbb N.$$
Then by Proposition \ref{scalarization},
$$t_nd_n\in Colev_{<^l}(F,\lambda_nq),~~~~~\mbox{for~all}~n\in\Bbb N.$$
Now, for each $n\in\Bbb N$ we set $x_n=t_nd_n$. Therefore, for each $n\in \Bbb N$, $x_n\in Colev_{<^l}(F,\lambda_nq)$, $t_n\to +\infty$ and $\frac{x_n}{t_n} \xrightarrow{\sigma} u$. Thus 
$$u\in {\limsup}_n^{\sigma,\infty} Colev_{<^l}(F,\lambda_nq)=\{0\}.$$
Then $u=0$, and $F$ satisfies the optimality condition (\ref{12}).\\
For the converse, let $F$ satisfies the optimality condition (\ref{12}). Therefore
$$(\Psi_F)^{\sigma g}(u)>M_F^q,~~~\mbox{for~all}~u\in K_{\sigma}^{\infty}\setminus\{0\}.$$
Suppose that $\lambda_n\downarrow M_F^q$ and $u\in {\limsup}_n^{\sigma,\infty} Colev_{<^l}(F,\lambda_nq)$. Thus there exist $x_{n_j}\in Colev_{<^l}(F,\lambda_{n_j}q)$ and $t_{n_j}\to +\infty$ such that $\frac{x_{n_j}}{t_{n_j}}\xrightarrow{\sigma} u$. If for each $j\in\Bbb N$, we set $d_j=\frac{x_{n_j}}{t_{n_j}}$, then $d_j\xrightarrow{\sigma} u$ and for each $j\in\Bbb N$, $t_{n_j}d_j\in Colev_{<^l}(F,\lambda_{n_j}q)$. Hence, by Proposition \ref{scalarization}, for each $j\in\Bbb N$, $\Psi(t_{n_j}d_j)\leq \lambda_{n_j}$. Thus
$$\liminf_{j\to\infty}\Psi_F(t_{n_j}d_j)\leq M_F^q.$$
Therefore
$$(\Psi_F)^{\sigma g}(u)=\inf\{\liminf_{n\to\infty}\Psi_F(t_nd_n): t_n\to +\infty, d_n \xrightarrow{\sigma} u\}=M_F^q.$$
Then $u=0$ and we conclude that
$${\limsup}_n^{\sigma,\infty} Colev_{<^l}(F,\lambda_nq)=\{0\}.$$
\qed
\end{proof}

By Corollary \ref{cor2}, Theorem \ref{thm1}, Example \ref{exp 6}(i) and Proposition \ref{Prop horizon}, we have the following corollary.
\begin{corollary}
Suppose that $E=\Bbb R^n$ and $Y=\Bbb R^m$ are equipped with the usual norm topology, $\sigma=s$ and $F:\Bbb R^n\rightrightarrows \Bbb R^m$ satisfies assumption $(A)$. Let $F$ be $l$-colevel closed  and for each $\lambda_n\downarrow M_F^q$,
$${\limsup}_n^{s,\infty} Colev_{<^l}(F,\lambda_nq)=\{0\}.$$
If $F(\Bbb R^n)$ is $P$-bounded, then $l$-$SWEff(F)$ is a nonempty set.
\end{corollary}
\begin{remark}\label{rem 3}
\begin{itemize}
\item[(i)] It follows that, if $F: E\rightrightarrows Y$ satisfies assumption $(A)$ and $q$-sgicc with respect to the $\sigma$ topology, there exists $\lambda>M_F^q$ such that $Colev_{<^l}(F,\lambda q)$ is $\sigma$-bounded. Hence, by the definition, $(Colev_{<^l}(F,\lambda q))^{\infty}_{\sigma}=0$. Then, by Proposition \ref{P. property qr}(i), $F$ satisfies the optimality condition (\ref{12}).
\item[(ii)] Let $F: E\rightrightarrows Y$ be $\sigma q$-srgi, then it is not necessarily $\sigma q$-srgi on $nB_E$ for every $n\in\Bbb N$, even if, $(E,\|.\|)$ is a finite dimensional Banach space and $\sigma=s$. Therefore Proposition \ref{first theorem} is not a consequence of Theorem \ref{thm1}.\\
For example suppose that $E=Y=\Bbb R$ are equipped with the usual norm topology, $\sigma=s$, $P=\Bbb R_+$ and $q=1$. Let $F:\Bbb R\rightrightarrows \Bbb R$ be defined as follows:
$$F(x)=\left\{
		\begin{array}{ll}
				[1, -x+\frac{1}{2}],~~~~~&x\leq -\frac{1}{2},\\
				
			[\frac{3}{4}x+	\frac{7}{8}, 2],~~~~~&-\frac{1}{2}<x<\frac{3}{2},\\
	
			\{x-\frac{3}{2}\},~~~~&x\geq \frac{3}{2}.\\
			\end{array}\right.$$
	Hence
	$$\Psi_F(x)=\left\{
		\begin{array}{ll}
				1,~~~~~& x\leq -\frac{1}{2},\\
				
			\frac{3}{4}x+	\frac{7}{8},~~~~~& -\frac{1}{2}<x<\frac{3}{2},\\
	
			x-\frac{3}{2},~~~~& x\geq \frac{3}{2}.\\
			\end{array}\right.$$
Therefore $F$ satisfies all the conditions of Proposition \ref{first theorem}. But, since $\Psi_F$ is not $s$-rgi on $B_{\Bbb R}$, thus by Proposition \ref{Prop.rgi}, $F$ is not $sq$-srgi on $B_{\Bbb R}$. Then we can not use Theorem \ref{thm1} for this map.
\end{itemize}
\end{remark}
In the following example we show that there exists a set-valued map $F$ that satisfies the optimality condition (\ref{12}) and the condition $(K_q^{set})$, but it does not satisfy $q$-sgicc. In fact, the map that is defined in the following example satisfies all the conditions of Theorem \ref{thm1}, but we can not use Proposition \ref{first theorem} for this map.
\begin{example}\label{e23}
Suppose that for $p>1$, $E=Y=\ell^p$ are equipped with the usual norm topology, $\sigma=w$ is the weak topology, $P=\ell_+^p:=\{x=(x_n)\in\ell^p:x_n\geq 0, \forall n\in\Bbb N\}$ and $q=(\frac{1}{2^n})$. Let $C:=\{x=(x_n)\in \ell^p:|x_n|\leq 1,~\forall n\in\Bbb N\}$. By \cite{M0}, $C$ is a convex, $w$-closed and an unbounded subset of $\ell^p$, such that $C^{\infty}_{w}=\{0\}$.\\
Let $F:\ell^p\rightrightarrows \ell^p$ be defined by
		$$F(x)=\left\{
		\begin{array}{ll}
				\{(\frac{1}{2^{n-1}})\},& x\in C^c,\\
			\{0\}	,&x=\{0,e_2\},\\
			\{(0,\frac{1}{(m+1)},\frac{1}{(m+1)^2},...)\},& x\in C\cap (mB_{\ell^p}\setminus \{(m-1)B_{\ell^p}\cup \{e_2\}\}),\\
			\end{array}\right.$$
where $m\in\Bbb N$, $0B_{\ell^p}:=\{0\}$ and $e_2=(0,1,0,0,...)$.
It's easy to see that
 $$\Psi_F(x)=\left\{
		\begin{array}{ll}
			0,& x\in C,\\
			2,& x\in C^c.
		\end{array}\right.$$
Therefore we have $M_F^q=0$ and
		$$Colev_{<^l}(F,\lambda q)=\left\{
		\begin{array}{ll}
			C,& 0<\lambda<2,\\
			\ell_p,& \lambda\geq 2.
		\end{array}\right.$$	
Hence by Corollary \ref{Cr. rq closed}, $F$ is a $wq$-srgi map. Also, $F$ does not satisfy $q$-sgicc with respect to the weak topology. But for all $u\in \ell^p\setminus \{0\}$ we have
$$F^G_{w}(u)=\inf\{\lambda: u\in (Colev_{<^l}(F,\lambda q))_{w}^{\infty}\}\geq 2>M_F^q=F^G_{w}(0).$$
Now assume that $(\lambda_n)\subseteq \Bbb R$, $\lambda_n\downarrow M_F^q=0$, $x_n\in Colev_{<^l}(F,\lambda_nq)\bigcap nB_{\ell^p}$, for each $n\in \Bbb N$ and $\|x_n\|\rightarrow +\infty$. Then $x_n\in C$ for all $n\in \Bbb N$. For every $n\in \Bbb N$ we set $w_{n}:=e_2$.
 Thus $F(w_n)\leq^l F(x_n)$
and $\frac{w_n}{\|w_n\|}=e_2\xrightarrow{w} e_2$. Therefore, $F$ satisfies condition $(K_q^{set})$. Then $F$ satisfies all the conditions of Theorem \ref{thm1} and $0\in l$-$SWEff(f)\neq \emptyset$.	
\end{example}

\section{Conclusions}
In this paper we have extended the notions of regular-global-inf function, global-inf-coercive-condition \cite{A1}, and $qx$-asymptotic function \cite{H1}, to set-valued maps. Also, we have introduced a new scalarization function for set-valued maps. As the main results, we have proved two new Weierstrass-type theorems for a noncontinuous set optimization problem. In fact, in the first main theorem we give an existence result of strict weakly $l$-efficient solution of (SOP), when the set-valued map is regular-global-inf and it satisfies global-inf-coercive-condition. Moreover, in the second main result we establish an existence result of strict weakly $l$-efficient solution of (SOP) by using the $qx$-asymptotic function, that is extended to set-valued maps. This contribution improves and extends various existence results in the literature.

\begin{acknowledgements}
The authors would like to thank the associate editor and reviewers for their constructive comments, which helped us to improve the paper.
\end{acknowledgements}



\end{document}